%% file: Math_Z_Version_2__1_.tex
\documentclass{amsart}

\usepackage{amssymb}
\usepackage{amsthm}
\usepackage{amsmath}
\usepackage{mathrsfs}			
\usepackage[usenames,dvipsnames]{color}
\usepackage[unicode]{hyperref}	
\hypersetup{urlcolor=blue, citecolor=blue, linkcolor=blue, colorlinks=true}		
\usepackage{geometry}
\usepackage{subfigure}
\usepackage{tikz-cd}
\usepackage[bf]{caption}
\usepackage{placeins}
\usepackage{mathtools}

\renewcommand{\phi}{\varphi}

\newcommand{\QQ}{\mathbf{Q}}
\newcommand{\ZZ}{\mathbf{Z}}
\newcommand{\RR}{\mathbf{R}}
\newcommand{\CC}{\mathbf{C}}

\newcommand{\cO}{\mathcal{O}}

\usepackage{cleveref}

\newcommand{\mf}[1]{\mathfrak{#1}}
\newcommand{\mc}[1]{\mathcal{#1}}

\newcommand{\gl}{\mathrm{GL}}
\newcommand{\GO}{\mathrm{GO}}

\renewcommand{\varepsilon}{\epsilon}

\DeclareMathOperator{\disc}{disc}

\DeclareMathOperator{\Tr}{Tr}

\DeclareMathOperator{\sh}{sh}

\DeclareMathOperator{\unit}{unit}
\DeclareMathOperator{\Cl}{Cl}

\DeclareMathOperator{\Log}{Log}

\renewcommand{\H}{\mf{H}}
\newcommand{\Sr}{\mc{S}}

\newcommand{\vect}[1]{\begin{pmatrix}#1\end{pmatrix}}

\newtheorem{theorem}{Theorem}[section]

\newtheorem{proposition}[theorem]{Proposition}

\numberwithin{equation}{section}

\hypersetup{pdfauthor={Robert Harron, Erik Holmes, Sameera Vemulapalli},pdftitle={D_p unit shapes},pdfkeywords={Algebraic number theory, quintic fields, D5, lattices, arithmetic statistics, equidistribution, geodesics, hypercycles}}

\newcommand{\Rcom}[1]{}

\begin{document}

\title{Shapes of Unit Lattices in $D_p$-number fields}

\author{Robert Harron}
\address{
Radix Trading\\
Chicago, IL \\
USA
}
\email{robert.harron@gmail.com}

\author{Erik Holmes}
\address{
Department of Mathematics\\
University of Toronto\\
Toronto, ON\\
Canada
}
\email{eholmes@math.toronto.edu}

\author{Sameera Vemulapalli}
\address{
Department of Mathematics\\
Harvard University\\
Cambridge, MA\\
USA
}
\email{vemulapalli@math.harvard.edu}

\date{\today}

\begin{abstract}
The unit group of the ring of integers of a number field, modulo torsion, is a lattice via the logarithmic Minkowski embedding. We examine the shape of this lattice, which we call the \emph{unit shape}, within the family of prime degree $p$ number fields whose Galois closure has dihedral Galois group $D_p$ and a unique real embedding. In the case $p = 5$, we prove that the unit shapes lie on a single hypercycle on the modular surface (in this case, the modular surface is the space of shapes of rank $2$ lattices). For general $p$, we show that the unit shapes are contained in a finite union of translates of periodic torus orbits in the space of shapes. 
\end{abstract}

\maketitle

\tableofcontents

\begin{figure}[h]\label{fig:D5_data}
	\includegraphics[scale=0.55]{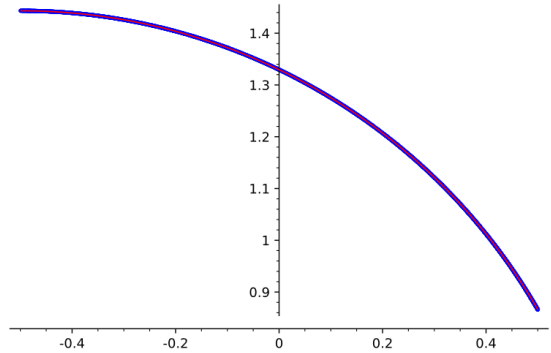}
	\caption{A $D_5$-extension with one real embedding has a rank $2$ unit lattice. For each of the 5422 such fields in the LMFDB (see \cite{LMFDB}), we computed a certain basis for the unit lattice and plotted the corresponding point in the upper half plane above. Observe a striking pattern: the resulting points lie on an arc of the circle $\left(x+\frac{1}{2}\right)^2+\left(y-\frac{1}{2\sqrt{3}}\right)^2=\left(\frac{2}{\sqrt{3}}\right)^2$.}
	\label{fig:D5_data}
\end{figure}

\section{Introduction}

The unit group of the ring of integers of a number field, modulo torsion, has the structure of a lattice via the logarithmic Minkowski embedding. This paper investigates the following general question: which lattices can arise this way? Let $p \geq 5$ be a prime. In this paper, we deal with the case of degree $p$ number fields with one real embedding whose Galois closure has dihedral Galois group $D_p$. Motivated by \Cref{fig:D5_data}, we ask:

\begin{center}
	\emph{Which lattices arise as unit lattices of $D_p$-extensions with one real embedding?}
\end{center}

(By standard abuse of notation in arithmetic statistics, a $D_p$-extension is a degree $p$ number field whose Galois closure has Galois group $D_p$). The main result of this paper, \Cref{maintheorem}, describes a finite union of translates of periodic torus orbits which contains the unit lattices of \emph{every} $D_p$-extension with one real embedding. When $p = 5$, Dirichlet's unit theorem shows that unit lattice has rank $2$, and the space of rank $2$ lattices up to scaling, rotation, and reflection is identified with the quotient $\gl_2(\ZZ)\backslash\mf{H}$ where $\mf{H}$ is the complex upper half plane. In this case, we prove that \emph{every} such unit lattice lie on the curve shown in \Cref{fig:D5_data}, and that this curve is a hypercycle on the upper half plane $\mf{H}$.

\subsection{Motivation}
\label{subsec:motivation}
We focus on the case of $D_p$-extensions for two reasons. First, when starting this project, we noticed that when $p = 5$, computational data from the LMFDB (the $L$-functions and modular forms database) seemed to imply that the unit lattices are constrained; see \Cref{fig:D5_data}. This is one of the simplest cases where the unit lattice seems to be constrained (unlike the case of totally real $S_3$-cubic extensions -- see \cite{DS unit shape}), so we seek to explain it. Second, in this case, we have the following fantastic result of Moser (see \cite[Proposition III.1]{Moser1979}) which explicitly describes the $\ZZ[D_p]$-module structure of the unit lattice of the Galois closure of our degree $p$ extension. It is this explicit description that is the key input in this paper. For a general group $G$ and a Galois $G$-extension, it is extremely difficult to determine the $\ZZ[G]$-module structure of the unit lattice.

\subsection{Setup}
Let $K$ be a degree $n$ number field. Let $\sigma_1,\dots,\sigma_r$ be the real embeddings of $K$ and $\tau_1, \overline{\tau_1},\dots, \tau_s, \overline{\tau_s}$ be the pairs of complex conjugate embeddings of $K$, where here $r + 2s = n$. Let $E_K$ be the group of units of $\cO_K$ modulo torsion. The \emph{logarithmic embedding}
	\begin{align*}
		\Log\colon E_K		&  	\longrightarrow \RR^{r + s}														\\
		 	u		&	\mapsto (\log|\sigma_1(u)|, \hdots, \log|\sigma_r(u)|, 2\log | \tau_1(u) |,\hdots,2\log | \tau_s(u) |)
	\end{align*}
 presents $E_K$ as a lattice of rank $r+s-1$. We say $E_K$ is the \emph{unit lattice} of $K$. See \cite[Chapter~1.7]{neukirch} for details.

In this paper, we study these lattices up to scaling, rotation, and reflection. The shape of a rank $n$ lattice $\Lambda$ is denoted $\sh(\Lambda)$ and is defined to be its equivalence class under scaling, rotation, and reflection. It is typically written as an element of the following double coset space:
	\[	\mathcal{S}_n \coloneqq \gl_n(\ZZ)\backslash \gl_n(\RR) / \GO_n(\RR)		\]
Under this equivalence, the rows of the matrix in $\gl_n(\RR)$ are the vectors which form a basis of the lattice. We call the shape of the unit lattice the \emph{unit shape}. (See \cite[Chapter~2]{milne} for details).

\subsection{Statement of results}\label{results}
For simplicity, we begin by stating the results for $p=5$. 

\subsubsection{The quintic case}
Recall that $\mathcal{S}_2$ can be identified with the quotient $\gl_2(\ZZ)\backslash \mf{H}$, where here $\mf{H}$ is the upper half plane. Let $x,y$ be coordinates for $\mf{H}$.

\begin{theorem}\label{hypercycle}
	Let $K$ be a $D_5$-extension with a unique real embedding. Then, the unit shape of $K$ lies on the arc of the circle
\begin{align*}
	\left(x+\frac{1}{2}\right)^2+\left(y-\frac{1}{2\sqrt{3}}\right)^2=\left(\frac{2}{\sqrt{3}}\right)^2
\end{align*}
from $\frac{1+i\sqrt{3}}{2}$ to $-\frac{1}{2}+i\frac{5}{2\sqrt{3}}$ in $\gl_2(\ZZ)\backslash \mf{H}$.
\end{theorem}

We now sketch a proof of \Cref{hypercycle}, which has four key steps. 
\begin{enumerate}
    \item For every such $D_5$-extension $K$ of $\QQ$ with Galois closure $L$, the action of $D_5$ on $E_L$ presents $E_L$ as a $\ZZ[D_5]$-module. This action restricts to an action of a smaller ring $\ZZ[\sigma + \sigma^{-1}]$ on $E_K$. (Here, $\sigma$ is a $5$-cycle in $D_5$).
    \item Via the logarithmic embedding, the action of $\ZZ[\sigma + \sigma^{-1}]$ on $E_K$ gives rise to an action of $\RR[\sigma + \sigma^{-1}]$ on $E_K \otimes \RR \simeq \RR^3$. 
    
    \item  Due to work of Moser (\cite{Moser1979}), $\ZZ[\sigma + \sigma^{-1}] \simeq \ZZ[\zeta_5 + \zeta_5^{-1}]$ and as a module over this ring, $E_K$ is isomorphic to a fractional ideal of $\QQ^+ \coloneqq \QQ(\zeta_5 + \zeta_5^{-1})$. Therefore, $E_K \otimes \RR$ and $\QQ^+ \otimes \RR$ are isomorphic as $\RR[\sigma + \sigma^{-1}] \simeq \RR[\zeta_5 + \zeta_5^{-1}]$-modules, but the quadratic forms on the two spaces are slightly different. 
    \item Now, for every number field $M$, the lattice shapes of all discrete and nonzero $\cO_M$-submodules of $M \otimes \RR$ in $M \otimes \RR$ forms a finite union of torus orbits in the moduli space of lattices. Therefore, the lattice shape of $E_K$ lies in a translate of this finite union of torus orbits. 
\end{enumerate}

Next, we say a few words about the arc mentioned in \Cref{hypercycle}, which has an interesting interpretation in the hyperbolic plane. Endow the upper half plane $\mf{H}$ with its usual hyperbolic geometry. A geodesic in this geometry is given by either an arc of a (Euclidean) half-circle whose center lies on the real axis or a segment of a vertical ray. It turns our that a connected component of the set of points equidistant from a geodesic is itself either an arc of a circle or a line segment, respectively. Such curves are (sometimes) called \textit{hypercycles}. Given a geodesic $\gamma$ whose diameter is the real interval $[A,B]$, the hypercycles of $\gamma$ are the arcs in $\mf{H}$ lying on circles whose intersection with $\RR$ is $\{A,B\}$. In the case that $\gamma$ is a segment of the vertical line, $x=A$, the hypercycles of $\gamma$ are segments of lines with non-zero slope that intersect the real axis at $x=A$.

In our case, let $\gamma$ be the geodesic arc in $\H$ connecting the points $\frac{1+i\sqrt{5}}{3}$ and $\frac{-1+i\sqrt{5}}{2}$. This is an arc of the circle centered at $-1/2$ of radius $\sqrt{5}/2$. Let $\psi$ denote the hypercycle a distance $\frac{1}{2}\log(5/3)$ from $\gamma$, given by the equation:
\begin{align*}
	\left(x+\frac{1}{2}\right)^2+\left(y-\frac{1}{2\sqrt{3}}\right)^2=\left(\frac{2}{\sqrt{3}}\right)^2, 
\end{align*}
with endpoints $\frac{1+i\sqrt{3}}{2}$ and $-\frac{1}{2} + i \frac{5}{2\sqrt{3}}$. 
This is the arc mentioned in the statement of \Cref{hypercycle}. Our experimental data suggests that the unit shapes are dense on this arc, but we do not know how to prove this. 

Why do the shapes of unit lattices lie on a hypercycle? We prove in \Cref{maintheorem} that for degree $p$ extensions with Galois group $D_p$ and $1$ real embedding, the shapes of unit lattices lie on a finite union of translates of periodic torus orbits. When $p = 5$, the shapes lie on a single translate of a periodic torus orbit. Moreover, this particular translate of this periodic torus orbit happens to be a hypercycle.

Moser's explicit description of the $\ZZ[D_5]$-module structure of $E_L$ involves the quadratic field $\QQ^+ = \QQ(\zeta_5 + \zeta_5^{-1})$, and thus we observe an interesting relationship between the regulator of $K$ and the norm form of $\QQ^+$. Let $A^+ = \ZZ[\zeta_5 + \zeta_5^{-1}]$ be the maximal order of $\QQ^+$, and let $N(x_0, x_1)$ be the norm form of $A^{+}$ with respect to the basis $\{1, \zeta_5+\zeta_5^{-1}\}$. A straightforward calculation shows that $N(x_0, x_1)= x_0^2 - x_0x_1 - x_1^2$, and that this indefinite quadratic form corresponds to the geodesic $\gamma$ in $\gl_2(\ZZ)\backslash \mf{H}$. Let $\sigma \in D_5$ be a $5$-cycle.

\begin{theorem}\label{5regulator}
Let $K$ be a $D_5$-extension with a unique real embedding $\iota$. Then:
\begin{enumerate}
    \item there exists a unit $u_0 \in E_K$ such that $\{u_0, u_1 = \sigma(u_0) + \sigma^{-1}(u_0)\}$ forms a basis for $E_K$; 
    \item and for any such unit $u_0$, the regulator of $K$ is equal to
	\[
        \lvert N(\log \lvert \iota(u_0) \rvert, \log \lvert \iota(u_1) \rvert) \rvert.	
    \]
\end{enumerate}
\end{theorem}

\noindent Part $(1)$ of \Cref{5regulator} follows from Moser's work, so our contribution is part $(2)$.

\subsubsection{Generalizing to all primes $p$}

We first introduce some notation needed to state our generalization to all $p$. Let $p \geq 5$ be a prime and $r = (p-1)/2$. In this case the unit lattice of a $D_p$-extension has rank $r$. Let $\mathcal{G}_r$ be the space of real symmetric positive definite $r\times r$ matrices up to scaling and let matrices $M \in \gl_r(\RR)$ act on $\mathcal{G}_r$ by conjugation, i.e. $MGM^T$ for $G \in \mathcal{G}_r$. Note that the space of shapes of rank $r$ lattices, $\Sr_r$, can be identified with $\gl_r(\ZZ)\backslash \mathcal{G}_r$.

In the case of general $p$, our unit shapes will be contained in finitely many translates of an orbit of a certain subgroup of $\gl_r(\RR)$ acting upon a certain element $G_{\unit} \in \mathcal{G}_r$. We now define $G_{\unit}$.  The cyclotomic field $\QQ(\zeta_p)$ contains the totally real subfield $\QQ^+ = \QQ(\zeta_p+\zeta_p^{-1})$, which has ring of integers $A^+ = \ZZ[\zeta_p + \zeta_p^{-1}]$. Since $\QQ^+$ is totally real, the ring of integers is a lattice with respect to its trace pairing. Let $G$ be the Gram matrix of the trace form of $A^+$ with respect to the $\ZZ$-basis $B= [1, \zeta_p + \zeta_p^{-1}, \hdots, \zeta_p^{r-1} + \zeta_p^{1-r}]$. Let $\widetilde{G}$ be the $r\times (r+1)$ matrix obtained by adjoining a \emph{zeroing column} to $G$: the column with $j$--th entry given by $-\sum_{i=1}^{r} G_{ji}$. Then define 
	\[	G_{\unit} \coloneqq \widetilde{G}\widetilde{G}^T.	\]
The matrix $G_{\unit}$ is the Gram matrix of the lattice spanned by the rows of $\widetilde{G}$. 

We now describe the subgroup of $\gl_r(\RR)$ that will act on $G_{\unit}$. Let $\mathcal{T}$ be the diagonal torus of $\gl_r(\RR)$. Let $j_1,\dots,j_{r}$ be the $r$ embeddings of $\QQ^+$ into the real numbers, ordered so that $j_i(\zeta_p + \zeta_p^{-1}) = \zeta_p^i + \zeta_p^{-i}$, and let $P$ be the matrix whose rows are given by the embeddings of $B$, so
\[
    P = 
\begin{bmatrix}
    1       & \dots & 1 \\
    j_1(\zeta_p + \zeta_p^{-1})       & \dots & j_r(\zeta_p + \zeta_p^{-1}) \\
    \vdots &  \ddots & \vdots \\
    j_1(\zeta_p^{r-1} + \zeta_p^{1-r})       & \dots & j_r(\zeta_p^{r-1} + \zeta_p^{1-r})
\end{bmatrix}.
\]
Let $\mathcal{T}_p = P\mathcal{T}P^{-1}$. 

We now describe the finitely many translates of the orbit of $\mathcal{T}_p$ on $G_{\unit}$. For every class $[\mathfrak{a}] \in \Cl(\QQ^+)$, pick a representative ideal $\mathfrak{a}$ and a $\ZZ$-basis $B_{\mathfrak{a}} = [v_1,\dots,v_{r}]$.  Let $M_{\mathfrak{a}}$ be the matrix sending $B$ to $B_{\mathfrak{a}}$, so $M_{\mathfrak{a}}B^T = B_{\mathfrak{a}}^T$. For every $[\mathfrak{a}] \in \Cl(\QQ^+)$, consider the image of $G_{\unit}$ under the action of $M_{\mathfrak{a}}\mathcal{T}_p$; this is a set in $\mathcal{G}_r$, and we call it $\mathcal{O}_{[\mathfrak{a}]}$. Because $\mathcal{S}_r = \gl_r(\ZZ) \backslash \mathcal{G}_r$, we obtain $[\mathcal{O}_{[\mathfrak{a}]}] \subseteq \mathcal{S}_r$ via the quotient. Then, the following generalization of \Cref{hypercycle} describes the space of unit shapes:

\begin{theorem}\label{maintheorem}
    Let $K$ be a $D_p$-extension with a unique real embedding and Galois group $D_p$. Then:
    \begin{enumerate}
        \item For each $[\mathfrak{a}] \in \Cl(\QQ^+)$, the torus orbit $\mathcal{O}_{[\mathfrak{a}]}$ is a translate of a periodic torus orbit in $\mathcal{G}_r$ and the quotient $[\mathcal{O}_{[\mathfrak{a}]}] \subseteq \mathcal{S}_r$ is independent of the choice of representative ideal $\mathfrak{a}$ and basis $B_{\mathfrak{a}}$;
        \item and the unit shape of $K$ is contained in $[\mathcal{O}_{[\mathfrak{a}]}]$ for some $[\mathfrak{a}] \in \Cl(\QQ^+)$.
    \end{enumerate} 
\end{theorem}

We note that the proof of \Cref{maintheorem} proceeds exactly as in the proof of \Cref{hypercycle}, with $5$ replaced with a prime $p$. The restriction to primes stems from our dependence on Moser's work (\cite{Moser1979}) describing the Galois module structure of units in $D_p$-extensions. The case $p > 5$ is conceptually the same as $p = 5$, but is more computationally difficult.

There is also a natural variant of the regulator relation for general $p$. Let $H$ be the hyperplane in $\RR^{r+1}$ whose coordinates sum to $0$. Zero index the rows of $\widetilde{G}$, and define the map
\[
    \varphi \colon \QQ^+ \otimes \RR \rightarrow H
\]
sending $1$ to the $0$-th row of $\widetilde{G}$ and $\zeta_p^i + \zeta_p^{-i}$ to the $i$-th row of $\widetilde{G}$ for $1 \leq i < r$. It is easily seen that $\varphi$ is an isomorphism of vector spaces. Via the $\Log$ map, we also have $H \simeq E_K \otimes \RR$.  Let $\mathfrak{a} = \varphi^{-1}(\Log(E_K))$; we will show that as a group $\mathfrak{a} \cong \ZZ^r$ and $A^+ \mathfrak{a} = \mathfrak{a}$, so $\mathfrak{a}$ is a fractional ideal of $A^+$. Let $N(\mathfrak{a})$ be the absolute value of the determinant of a matrix sending a basis of $A^+$ to a basis of $\mathfrak{a}$. The following is a generalization of \Cref{5regulator}.

\begin{proposition}\label{pregulator}
Let $K$ be a $D_p$-extension with a unique real embedding. Let $\mathfrak{a}$ be obtained as above. Then the regulator of $K$ is given by
		\[	R_K =  N(\mathfrak{a}) \disc(\QQ^+).	\]
\end{proposition}

The proof of \Cref{pregulator} is immediate from the fact that the determinant of $G$ is $\disc(\QQ^+)$ and $\varphi$ is an isomorphism of vector spaces. \Cref{maintheorem} specializes to \Cref{hypercycle} when $p = 5$. In this case the cyclotomic field $\QQ(\zeta_5)$ has totally real subfield $\QQ(\zeta_5+\zeta_5^{-1}) = \QQ(\sqrt{5})$, which has trivial class number. The ring of integers of the totally real field is $A^+ = \ZZ[\zeta_5 + \zeta_5^{-1}]$, and an integral basis of $A^+$ is $B = \{1, \zeta_5+\zeta_5^{-1}\}$. The Gram matrix of the trace form of $A^+/\ZZ$ is:
\[
    G = \begin{pmatrix} 2 & -1 \\ -1 & 3\end{pmatrix}.
\]
Thus, computing the Gram matrix we obtain:
\[
    G_{\unit} = \widetilde{G}\widetilde{G}^T = \begin{pmatrix} 6 & -3 \\ -3 & 14\end{pmatrix}.
\]
Note that $G_{\unit}$ corresponds to the point $\left(-\frac{1}{2}, \frac{5}{2\sqrt{3}}\right)$ in $\mf{H}$. Letting $\eta_1 \coloneqq \zeta_5 + \zeta_5^{-1}$ and $\eta_2 \coloneqq \zeta_5^2 + \zeta_5^{-2}$, we have the embedding matrix $P$:
    \[  P = \begin{pmatrix} 1 & 1 \\ \eta_1 & \eta_2 \end{pmatrix}.\]
So we obtain:
    \begin{align*}
      \mathcal{T}_5     &   = P \mathcal{T} P^{-1}                                \\
                        &   = \begin{pmatrix} 1 & 1 \\ \eta_1 & \eta_2 \end{pmatrix}\begin{pmatrix} a & 0 \\ 0 & b \end{pmatrix} \frac{1}{\eta_2-\eta_1}\begin{pmatrix} \eta_2 & -1 \\ -\eta_1 & 1 \end{pmatrix} \\
                        &   =  \frac{1}{\eta_2 - \eta_1} \begin{pmatrix} a\eta_2 - b\eta_1 & b-a \\ b-a & b\eta_2 - a\eta_1 \end{pmatrix}.
\end{align*}   
A computation shows that this group flows $G_{\unit}$ along the specified arc as $a$ and $b$ vary.

\subsection{Overview of previous work}
There is little known about the shapes of unit lattices, but here we include a brief overview of the recent work in the area.

\begin{itemize} 
    \item Ofir David and Uri Shapira, in \cite{DS unit shape}, examine the unit shape of totally real cubic fields and conjecture that this set is dense in $\mathcal{S}_2$: they provide evidence using a certain explicit family of cubic fields.

    \item Jose Miguel Cruz Rangel, in \cite{cyclic unit shapes}, investigates the unit shape of cyclic number fields of degree $\leq 7$. In particular, he studies the well-roundedness of these lattices and determines a parametrizing space of cyclic unit lattices. 

    \item Fernando Azpeitia-Tellez, Christopher Powell, and Shahed Sharif, in \cite{biquad unit shapes}, study the geometry of unit lattices of biquadratic fields, $K= \QQ(\sqrt{D_1}, \sqrt{D_2})$. They exploit the Galois module structure to determine when the unit lattice is generated by pairwise orthogonal vectors. This result can be interpreted to show that for a particular family of biquadratic fields, the unit shapes lie on restricted subspaces (more precisely, that they lie on a torus orbit).  

    \item Nguyen-Thi Dang, Nihar Gargava, and Jialun Li, in \cite{dang-gargava-li} show that the unit shape of orders in totally real cubic fields are \emph{dense} in $\mathcal{S}_2$.
        
    \item Independently, Emilio Corso and Federico Rodriguez Hertz, in \cite{corso-hertz}, show that the unit shape of orders in totally real cubic fields are \emph{unbounded} in $\mathcal{S}_2$.

    \item Sergio Ricardo Zapata Ceballos, Sara Chari, Fatemeh Jalalvand, Rahinatou Yuh Njah Nchiwo, Kelly O'Connor, Fabian Ramirez and the two named authors of this paper, in \cite{rnt} study the geometry of unit lattices of $D_4$-quartic fields with signature $(2,1)$. Namely, they show that the shapes in $\mathcal{S}_2$ arising from unit lattices of such fields are transcendental and the set of limit points is nonempty. 
\end{itemize}

\subsection{Cryptographic context}\label{contextsection}

We note that although the shape is the central invariant in this paper, it is also natural to consider other geometric properties of these lattices. For example one may consider well-roundedness, orthogonality, or kissing number. Understanding the geometry of these lattices is an important question in number theory and has applications to cryptography. In the cryptographic setting, lattice-based schemes -- which have garnered significant attention in the post-quantum cryptography space -- rely on the hardness of certain geometric problems, such as the Shortest Vector Problem (SVP) and the Closest Vector Problem (CVP), both of which are profoundly influenced by the underlying lattice structure. Notably, in cyclotomic fields, the structure of the unit lattice was exploited to defeat the security of the SOLILOQUY cryptosystem, \cite{sililoquy}. This case underscores how subtle geometric and algebraic features of lattices can directly impact the security of cryptographic protocols and suggests that a deeper understanding of the geometry of lattices is critical not only for theoretical advancements but also for practical applications in secure communications.

Existing results on unit shapes, and most on shapes of integral lattices, have focused on low-degree number fields, making them impractical for cryptographic applications. This paper, however, studies unit shapes in arbitrarily large rank. So, while our results are primarily theoretical, they highlight restrictions in the shapes that may be of relevance to cryptographic research, where a nuanced understanding of the geometry could help identify both opportunities and vulnerabilities.

\subsection{Acknowledgments} 
The authors would like to thank John Voight, whose suggestion to explore shapes of unit lattices inspired the computations which led to this project. We'd also like to thank Ofir David,  Santiago Arango-Piñeros, and Uri Shapira for comments on an earlier version of this draft. We are also grateful to the referees for helpful comments. The third author was funded by the NSF under grant number DMS2303211.

\section{Preliminaries}
\label{sec:pre}
In this section, we discuss results from \cite{Moser1979} which describe the Galois module structure of the unit lattice. Let $p \geq 5$ be a prime. Let $K$ be a number field of degree $p$ with a unique real embedding $\iota$ whose Galois closure $L$ has dihedral Galois group $D_p$. By Dirichlet's unit theorem, the rank of $E_K$ is $(p-1)/2$; let $r = (p-1)/2$. Let $L$ be the Galois closure of $K$. Let $\tau$ denote the unique nontrivial element of $D_p$ such that $K=L^\tau$. Let $\sigma$ denote a $p$-cycle in $D_p$ so $D_p=\langle\sigma,\tau\rangle$ and $\tau\sigma\tau=\sigma^{-1}$. 

Let $\zeta_p$ be a primitive $p$--th root of unity and let $\QQ^+\coloneqq\QQ(\zeta_p+\zeta_p^{-1})$. It is the maximal totally real subfield of $\QQ(\zeta_p)$. Let $A = \ZZ[\zeta_p]$ and $A^+ = \ZZ[\zeta_p+\zeta_p^{-1}]$ denote the ring of integers of $\QQ(\zeta_p)$ and $\QQ^+$, respectively. Let $\mf{P}$ be the unique prime of $\QQ(\zeta_p)$ above $p$ and let $\mf{p}=\mf{P}\cap A^+$.

Following \cite{Moser1979}, we endow $A$ with the structure of a $\ZZ[D_p]$-module; let $\sigma$ acting by multiplication by $\zeta_p$ and let $\tau$ act by complex conjugation. Now, Proposition~III.1 of \cite{Moser1979} states that
\[
	E_L\cong\mf{a}\mf{P}^e\text{ for some $e\in\{0,1\}$ and some fractional ideal $\mf{a} \subseteq \QQ^+$}
\]
as a $\ZZ[D_p]$-module. So, if $e = 0$ then $E_L \simeq \mathfrak{a}A$ and if $e = 1$ then $E_L \simeq \mathfrak{a}\mathfrak{P}$. Now, Proposition~I.2 of \cite{Moser1979} says that $E_K=E_L^\tau$; note that this proposition is nontrivial as $E_L$ (resp. $E_K$) is the quotient of the unit group by roots of unity. Therefore, under the isomorphism above, we get
\[
    E_K=E_L^\tau \cong (\mf{a}\mf{P}^e)^{\tau} = \mf{a}\mf{p}^e 
\]
and the latter is a fractional ideal of $\QQ^+$. Now, we endow this identification with more structure. Observe that both $E_K$ and $\mf{a}\mf{p}^e$ are $\ZZ[\sigma + \sigma^{-1}]$-modules and we obtain an isomorphism $E_K \simeq \mf{a}\mf{p}^e$ of $\ZZ[\sigma + \sigma^{-1}]$-modules. It is this description of $E_K$ as a $\ZZ[\sigma + \sigma^{-1}]$-module that is crucial to our proof below.


\section{Proof of \Cref{hypercycle} and \Cref{5regulator}}
This section is devoted to the proof of \Cref{hypercycle} and \Cref{5regulator}. Let $K$ be a $D_5$-extension with unique real embedding $\iota$. Let $\sigma$ be a $5$-cycle in $D_5$. Let $E_K$ denote the unit group of the ring of integers modulo torsion. Let $S_{\ZZ} = \ZZ[\sigma + \sigma^{-1}]$. 

\subsection{The structure of $E_K$}
By the exposition in \Cref{sec:pre}, since $\QQ(\zeta_5)$ has trivial class number, it immediately follows that $E_K \simeq A^+$ as an $S_{\ZZ}$-module. Note that a basis for $A^+$ is $\{1, \zeta_5 + \zeta_5^{-1}\}$, and $(\sigma + \sigma^{-1})(1) = \zeta_5 + \zeta_5^{-1}$. So, for any isomorphism $\varphi \colon A^+ \rightarrow E_K$ of $S_{\ZZ}$-modules, $\{\varphi(1), (\sigma + \sigma^{-1})(\varphi(1))\}$ is a basis for $E_K$.

\subsection{The logarithmic embedding}
Now, we fix an extension of the real embedding $\iota$ to $L$. The set of embeddings of $L$ into $\CC$ is $\iota\circ g$ for $g\in D_5$. Since $\tau$ acts trivially on $K$, the set of embeddings of $K$ are given by $\iota \circ \sigma^i$, for $0\leq i \leq 4$. We denote each embedding by $\iota_i \coloneqq \iota \circ \sigma^i$. Note that the pairs $(\iota_1, \iota_4)$ and $(\iota_2, \iota_3)$ are pairs of complex conjugate embeddings because $\tau$ is complex conjugation. Suppose the logarithmic embedding $\Log \colon E_K \rightarrow \RR^3$ is given by sending 
\[
    u \rightarrow (\log |\iota_0(u)|, 2\log |\iota_1(u)|, 2 \log |\iota_2(u)|).
\]
The map $\Log \colon E_K \rightarrow \RR^3$ is also a map of $S_{\ZZ}$-modules; $S_{\ZZ}$ acts on $\RR^3$ via $(\sigma + \sigma^{-1})(a_0, a_1, a_2) = (a_1, 2a_0 + a_2, a_1 + a_2)$.

\subsection{Proof of \Cref{hypercycle}}
Choose an isomorphism $\varphi \colon A^+ \simeq E_K$ as $S_{\ZZ}$-modules. Let $u_0 = \varphi(1)$ and $u_1 = (\sigma + \sigma^{-1})(u_0)$. Then $\{u_0,u_1\}$ is a basis of $E_K$. 
Write 
\[
    \Log(u_0) = (a_0, a_1, -a_0-a_1)
\]
for real numbers $a_0,a_1$. Then:
\[
    \Log(u_1) = \Log(\sigma(u_0)) + \Log(\sigma^4(u_0)) = (a_1, a_0 - a_1, -a_0).
\]
\noindent The Gram matrix of the vectors $\{-\Log(u_1),\Log(u_0)\}$ is:
 	\[	G =
			\begin{pmatrix}
				2(a_0^2 - a_0a_1 + a_1^2)	&	a_1^2 - 3a_0a_1 - a_0^2						\\
				a_1^2 - 3a_0a_1 - a_0^2		&	2(a_0^2 + a_0a_1 + a_1^2 )
			\end{pmatrix}
 	\]
 Given a rank 2 lattice with Gram matrix $G = (G_{ij})$ we find the associated point in the upper half plane by letting
 	\begin{align*}
		x 	= \frac{G[0,1]}{G[0,0]},	&\;\;\;		y	= \sqrt{\frac{G[1,1]}{G[0,0]} - x^2}
	\end{align*}
Now, a straightforward computation shows $(x,y)$ lies on the hypercycle $\psi$ and proves \Cref{hypercycle}.
 
 \subsection{Proof of \Cref{5regulator}}
 Now, we have the following result which allows us to compute the regulator of $K$ using the norm form of $A^+$ and the \emph{logarithmic units} of $K$. Let $N(x_0, x_1)\in \ZZ[x_0, x_1]$ be the binary quadratic form given by the norm form of $A^+$ in the basis $\{1, \eta_1= \zeta + \zeta^{-1}\}$. Let $u_0$ be a unit of $K$ such that $\{u_0, u_1 = (\sigma + \sigma^{-1})u_0\}$ is a basis for $E_K$. Then we will show that $\lvert N(\log \lvert \iota(u_0) \rvert, \log \lvert \iota(u_1) \rvert) \rvert$ is equal to the regulator of $K$.
 
First, an easy computation shows that $N(x_0, x_1) = x_0^2  - x_0x_1 - x_1^2$. Letting $u_0, u_1$ be as above, we obtain:
\[
    \begin{bmatrix}
           \Log(u_0) \\           
           \Log(u_1)
    \end{bmatrix} = 
    \begin{pmatrix}
		a_0	&	a_1	& -a_0 - a_1	\\
		a_1 	&	a_0 - a_1	& -a_0
	\end{pmatrix}
\]
The regulator of $K$ is the determinant of any $2\times 2$ submatrix, which is easily seen to be $a_0^2  - a_0a_1 - a_1^2$.

 
\section{Proof of \Cref{maintheorem}} 
This section is devoted to the proof of \Cref{maintheorem}. Let $K$ be a degree $p$ number field with unique real embedding, Galois group $D_p$, and Galois closure $L$. As before, $E_K$ denotes the unit group of the ring of integers modulo torsion. Let $S_\RR = \RR[\sigma + \sigma^{-1}]$ and let $S_{\ZZ} = \ZZ[\sigma + \sigma^{-1}]$. Let $r = (p-1)/2$ and let $\eta_i = \zeta_p^i + \zeta_p^{-i}$ for $1 \leq i \leq r$. 

\subsection{Sketch of the proof of \Cref{maintheorem}}
We split the proof into steps:
\begin{enumerate}
    \item First, we'll show that $E_K \simeq \mathfrak{b}$ as an $S_{\ZZ}$-module for some fractional ideal $\mathfrak{b}$ of $\QQ^+$. 
    \item Next, let $H$ be the hyperplane in $\RR^{r+1}$ whose coordinates sum to zero. \begin{enumerate}
        \item We'll describe the log map $\Log \colon E_K \otimes \RR \rightarrow H$.
        \item We'll next set up some isomorphisms 
\[
    E_K \otimes \RR \cong H \cong \QQ^+ \otimes \RR \cong \RR^r,
\]
of $S_{\RR}$-modules. The isomorphism $E_K \otimes \RR \rightarrow H$ will just be the $\Log$ map, and $\QQ^+ \otimes \RR \rightarrow \RR^r$ will be the Minkowski embedding. We'll let the isomorphism $H \cong \QQ^+ \otimes \RR$ be arbitary for now.
    \end{enumerate}

    \item Next, because $E_K \simeq \mathfrak{b}$ as an $S_{\ZZ}$-module, the image of $E_K$ inside $\RR^r$ is isomorphic to $\mathfrak{b}$ as an $S_{\ZZ}$-module. We'll then classify the $S_{\ZZ}$-modules isomorphic to $\mathfrak{b}$ inside of $\RR^r$.
    \item We'll use this description of $S_{\ZZ}$-modules to describe the possible shapes of $E_K$ in terms of $\varphi$.
    \item Finally, we'll make an explicit choice of $\varphi$.
    \item We'll substitute this choice of $\varphi$ into Step $(4)$ to explicitly describe the possible shapes of $E_K$.
\end{enumerate}

\subsection{Step $(1)$: the structure of $E_K$}
By the exposition in \Cref{sec:pre}, we have that $E_L \simeq \mathfrak{a}\mathfrak{P}^{e}$ as a $\ZZ[D_p]$-module for some $e \in \{0,1\}$. Under this isomorphism:
\[
    E_K = E_L^{\tau} \simeq (\mathfrak{a}\mathfrak{P}^{e})^{\tau} = \mathfrak{a}\mathfrak{p}^{e}.
\]
Let $\mathfrak{b} = \mathfrak{a}\mathfrak{p}^{e}$, and as before, this is an isomorphism of $S_{\ZZ}$-modules. Tensoring with $\RR$, we see that
\[
    E_K \otimes \RR \simeq \QQ^+ \otimes \RR
\]
as an $S_\RR$-module. 

\subsection{Step $(2)$, part $(a)$: the logarithmic embedding}
We fix an extension of the real embedding $\iota$ to $L$, which we still denote $\iota$. Then, the set of embeddings of $L$ into $\CC$ are $\iota\circ g$ for $g\in D_p$. Since $\tau$ preserves $K$, the set of embeddings of $K$ into $\CC$ is $\iota\circ\sigma^i$ for $0\leq i\leq p-1$. Let $\iota_i\coloneqq\iota\circ\sigma^i$. Because $\tau$ is complex conjugation of the roots in the complex plane, we have that $\iota_i$ and $\iota_{p-i}$ are a pair of complex conjugate embeddings $K\rightarrow\CC$ for $1\leq i\leq p-1$.

Now suppose the logarithmic embedding $\Log \colon E_K \rightarrow \RR^{r+1}$ is given by sending 
\[
    u \rightarrow (\log |\iota_0(u)|, 2\log |\iota_1(u)|, \ldots, 2 \log |\iota_{r}(u)|).
\]
Let $H$ be the hyperplane inside $\RR^{r+1}$ whose coordinates sum to zero. Because $(\sigma + \sigma^{-1})(K) = K$ and $E_K \otimes \RR \simeq H$, the hyperplane $H$ naturally has the structure of a $S_{\RR}$-module. Moreover, the logarithmic embedding shows that $E_K \otimes \RR \simeq H$ as an $S_{\RR}$-module.

\subsection{Step $(2)$, part $(b)$: the isomorphisms}
Pick an isomorphism $\varphi \colon \QQ^+ \otimes \RR \rightarrow H$ of $S_{\RR}$-modules. We have the series of $S_{\RR}$-module isomorphisms:
\[
    E_K \otimes \RR \rightarrow H \cong_{\varphi} \QQ^+ \otimes \RR \rightarrow \RR^{r}
\]
where $\QQ^+ \otimes \RR \rightarrow \RR^r$ is the Minkowski embedding sending $1 \mapsto (1,\ldots,1)$ and $\eta_1 \mapsto (\eta_1,\eta_2,\dots,\eta_{r})$. 

\subsection{Step $(3)$: $S_{\ZZ}$-modules isomorphic to $\mathfrak{b}$ inside $\RR^r$}
Recall that $E_K \simeq \mathfrak{b}$ as an $S_{\ZZ}$-module. Choose a $\ZZ$-basis $B_{\mathfrak{b}} = \{b_1,\dots,b_r\}$ of $\mathfrak{b}$. Then the subsets of $\QQ^+ \otimes \RR$ that are isomorphic to $\mathfrak{b}$ as an $S_{\ZZ}$-module are precisely fractional ideals in the same ideal class, i.e. they have a $\ZZ$-basis of the form $\{ xb_1,\dots, xb_r\}$ for $x \in (\QQ^+ \otimes \RR)^{\times}$. In $\RR^{r}$, these subsets have a $\ZZ$-basis of the form $\{(\alpha_1,\dots,\alpha_r)\cdot b_1, \dots, (\alpha_1, \dots,\alpha_r)\cdot b_r\}$ for $\alpha_1,\dots,\alpha_r \in \RR^{\times}$, where here the multiplication is coordinate-wise. 

\subsection{Step $(4)$: describing possible shapes of $E_K$ in terms of $\varphi$}

Recall that $\mathcal{G}_r$ is the space of real symmetric positive definite matrices up to scaling, and $\mathcal{S}_r = \gl_r(\ZZ)\backslash \mathcal{G}_r$. 

Let $G_{\mathfrak{b}}$ be the Gram matrix of the lattice spanned by $\{\varphi(b_1),\dots,\varphi(b_r)\}$. We are now ready to write down the group action.
Let $j_1,\dots,j_r$ be the $r$ embeddings $\QQ^+ \rightarrow \RR$ giving rise to the Minkowski embedding above, ordered so that $j_i(\eta_1) = \eta_i$. Define $P_{\mathfrak{b}}$ to be the matrix whose rows are given by the Minkowksi embedding above:
 \[	P_{\mathfrak{b}} = 
		\begin{pmatrix}
				j_1(b_1)		& \dots &	j_r(b_1)	\\
                \dots	& \dots &	\dots	\\
				j_1(b_r)	&\dots &	j_r(b_r)
			\end{pmatrix}.
 	\]
Let $\mathcal{T}$ be the diagonal torus of $\gl_r(\RR)$. By the discussion above, there exists a basis of $E_K$ such that the Gram matrix with respect to this basis is contained in the orbit of $P_{\mathfrak{b}}\mathcal{T}P_{\mathfrak{b}}^{-1}$ acting on $G_{\mathfrak{b}}$; we consider this orbit to be inside $\mathcal{G}_r$. Note that after the quotient to $\mathcal{S}_r$, this orbit depends only on the class $[\mathfrak{b}] \in \Cl(\QQ^+)$. Because the orbit in the quotient arises from an ideal class in a real quadratic field, it is a translate of a periodic torus orbit. 

\subsection{Step $(5)$: choosing a nice $\varphi$}
\label{subsec:phi-choice}
We now choose a nice choice of $\varphi$. Recall that $A^+$ has $\ZZ$-basis $B = \{1,\eta_1,\dots,\eta_{r-1}\}$. Let $G$ be the Gram matrix of the trace form of $A^+/\ZZ$ with respect to $B$. Let $\widetilde{G}$ be the $r\times (r+1)$ matrix obtained by adjoining a \emph{zeroing column} to $G$: the column with $j$--th entry given by $-\sum_{i=1}^{r} G(j,i)$. So:

\[
	\widetilde{G} \coloneqq	\begin{pmatrix}
				\Tr(1)	& \Tr(\eta_1)	& \dots & \Tr(\eta_{r-1}) &	-\sum_{i = 0}^{r-1}\Tr(\eta_i)	\\
                \Tr(\eta_{1}) & \Tr(\eta_1^2)	&\dots & \Tr(\eta_1 \eta_{r-1} )& -\sum_{i = 0}^{r-1}\Tr(\eta_{1}\eta_i) \\
                \dots	& \dots &	\dots	\\
                \Tr(\eta_{r-2}) & \Tr(\eta_1 \eta_{r-2})	&\dots & \Tr(\eta_{r-2}\eta_{r-1} )& -\sum_{i = 0}^{r-1}\Tr(\eta_{r-2}\eta_i) \\
                
				\Tr(\eta_{r-1}) & \Tr(\eta_1 \eta_{r-1})	&\dots & \Tr(\eta_{r-1}^2 )& -\sum_{i = 0}^{r-1}\Tr(\eta_{r-1}\eta_i)
			\end{pmatrix},
\]
where here $\Tr$ is the trace map from $\QQ^+$ to $\QQ$. For the convenience of the reader, we include a computation of the trace matrix $G$ below:
\[
		G = \vect{\frac{p-1}{2} & -1 & -1 & \cdots & \cdots &-1\\
			-1 &p-2&-2&\cdots& \cdots &-2\\
			\vdots&-2&p-2&-2& \cdots &-2\\
			\vdots&\vdots&&\ddots&&\vdots\\
			-1&-2&&&&p-2},
\]
We omit the proof as it is not directly needed for our results.

Now, define the map
\[
    \varphi \colon \QQ^+ \otimes \RR \rightarrow H
\]
sending the basis $B = \{1,\eta_1,\dots,\eta_{r-1}\}$ to the rows of $\widetilde{G}$. For notational simplicity, we $0$-index the rows and columns of $\widetilde{G}$ in \Cref{p:equivariant}. 

\begin{proposition}
\label{p:equivariant}
The map $\varphi$ is an isomorphism of $S_{\RR}$-modules.
\end{proposition}
\begin{proof}
It suffices to check that 
\[
    (\sigma^i + \sigma^{-i})(\varphi(1)) = \varphi(\eta_i)
\]
for every $1 \leq i < r$.  This is equivalent to saying that $(\sigma^i + \sigma^{-i})$ applied to the $0$--th row of $\widetilde{G}$ is the $i$--th row. It suffices to check the entries $\widetilde{G_{ij}}$ for $1 \leq i < r$ and $0 \leq j < r$, because the rows of $\widetilde{G}$ sum to zero. Again, we do the proof in cases. For $v \in \RR^{r+1}$ and $0 \leq j \leq r$, let $v[j]$ denote the entry in the $j$--th component. Let $\widetilde{G_0},\dots,\widetilde{G_r}$ be the rows of $\widetilde{G}$. 

In this notation, we are interested in showing that
\[
    \bigg((\sigma^i + \sigma^{-i})\widetilde{G_0}\bigg)[j] = \widetilde{G_i}[j]
\]
for all $1 \leq i < r$ and $0 \leq j < r$. By convention set $\eta_0 \coloneqq 1$ and $\eta_k = \zeta_p^k + \zeta_p^{-k}$ for any integer $k$ such that $p \neq k$. First notice that $\widetilde{G}_i[r+1] = \Tr(\eta_i\eta_{r+1})$. Now we split the proof into three cases.

\noindent \textbf{Case 1: $j = 0$ and $1 \leq i < r$.} We have:
\begin{align*}
\widetilde{G_i}[j] &= \Tr(\eta_i) \\
&= \frac{1}{2}\Tr(\eta_i) + \frac{1}{2}\Tr(\eta_{-i}) \\
&= \bigg((\sigma^i + \sigma^{-i})\widetilde{G_0}\bigg)[j]
\end{align*}

\noindent \textbf{Case 2: $0 < j < r$ and $1 \leq i < r$ and $i \neq j$.} We have:
\begin{align*}
\widetilde{G_i}[j] &= \Tr(\eta_i \eta_j) \\
&= \Tr(\zeta^{i+j}+\zeta^{i-j}+\zeta^{-(i-j)}+\zeta^{-(i+j)}) \\
&= \Tr(\eta_{i + j}) + \Tr(\eta_{i - j}) \\
&= \bigg((\sigma^i + \sigma^{-i})\widetilde{G_0}\bigg)[j]
\end{align*}

\noindent \textbf{Case 3: $0 < i = j < r$.} We have:
\begin{align*}
\widetilde{G_i}[j] &= \Tr(\eta_i^2) \\
&= \Tr(\eta_{2i}+2) \\
&= \Tr(\eta_{2i}) + 2\Tr(1) \\
&= \bigg((\sigma^i + \sigma^{-i})\widetilde{G_0}\bigg)[j]
\end{align*}
\noindent
\end{proof}

\subsection{Step $(6)$: describing possible shapes of $E_K$ using the given $\varphi$}
Let $\varphi$ be given by $\widetilde{G}$. Let $M_{\mathfrak{a}}$ be the matrix sending $B = \{1,\eta_1,\dots,\eta_{r-1}\}$ to $B_{\mathfrak{a}} = \{b_1,\dots,b_r\}$, so $M_{\mathfrak{a}}B^T = B_{\mathfrak{a}}^T$. Recall that
\[
    P = 
\begin{bmatrix}
    1       & \dots & 1 \\
    j_1(\eta_1)       & \dots & j_r(\eta_1) \\
    \vdots &  \ddots & \vdots \\
    j_1(\eta_{r-1})       & \dots & j_r(\eta_{r-1})
\end{bmatrix}.
\]
where here $j_1,\dots,j_r$ are the $r$ embeddings of $\QQ^+$ into the real numbers corresponding to the Minkowski embedding above. Thus:
\[
    M_{\mathfrak{b}}P = P_{\mathfrak{b}}
\]
So:
\[
    P_{\mathfrak{b}}\mathcal{T}P_{\mathfrak{b}}^{-1} = M_{\mathfrak{b}}P\mathcal{T}P^{-1}M_{\mathfrak{b}}^{-1}.
\]
So the orbit of $P_{\mathfrak{b}}\mathcal{T}P_{\mathfrak{b}}^{-1}$ on $G_{\mathfrak{b}}$ is equal to the orbit of $M_{\mathfrak{b}}P\mathcal{T}P^{-1}$ on $M_{\mathfrak{b}}^{-1}G_{\mathfrak{b}}M_{\mathfrak{b}}$. Finally, observe that $M_{\mathfrak{b}}P\mathcal{T}P^{-1} = M_{\mathfrak{b}}\mathcal{T}_p$ and $M_{\mathfrak{b}}^{-1}G_{\mathfrak{b}}M_{\mathfrak{b}} = G_{\unit}$. So, there exists a basis of $E_K$ such that the Gram matrix with respect to this basis is contained in the orbit of $M_{\mathfrak{b}}\mathcal{T}_p$ on $G_{\unit}$, completing the proof of \Cref{maintheorem}.

\include{Bibliography}

\end{document}

%% file: Bibliography.tex
%
%